\documentclass[12pt, a4paper,reqno]{amsart}
\pdfoutput=1
\usepackage{tikz,genyoungtabtikz,enumitem,rotating,latexsym,bm,stmaryrd}
\usetikzlibrary{positioning,intersections}
\usepackage{amsmath,amsthm,amsfonts,amssymb,mathrsfs,pb-diagram}
\usepackage[pagebackref=false,bookmarks=true,colorlinks=true,linktoc=page,citecolor=darkgreen,linkcolor=blue,urlcolor=mediumblue]{hyperref}
\usepackage{caption}
\usepackage{mathtools}
 \usepackage[a4paper,margin=1.2in]{geometry}
\usepackage{cleveref} 
\usepackage{tikz,enumitem,rotating,latexsym,bm,stmaryrd}
\usepackage{mathdots}
\usetikzlibrary{positioning,intersections,shapes.geometric,calc,decorations.pathreplacing}
\usepackage{soul}
\usepackage{xcolor}
\definecolor{mediumblue}{rgb}{0.0, 0.0, 0.8}
\colorlet{darkgreen}{green!50!black}

\renewcommand{\geq}{\geqslant}

\renewcommand{\leq}{\leqslant}
\renewcommand{\trianglerighteq}{\trianglerighteqslant}

\tikzset{wei/.style=
{red,double=red,double
distance=1pt}}

\usepackage{ltxtable}
\usepackage{tabularx}

\tikzset{wei2/.style={red,double=red,double
distance=1pt}}

\allowdisplaybreaks
\numberwithin{equation}{section}
\usepackage{scalefnt}

\newtheorem{thm}{Theorem}[section]
\newtheorem{cor}[thm]{Corollary}

\newtheorem*{ack}{Acknowledgements}

\newtheorem*{prop*}{Proposition}
\newtheorem*{thmA*}{Theorem A}
\newtheorem*{thmB*}{Theorem B}
\newtheorem*{thmC*}{Theorem C}

\newtheorem*{cor*}{Corollary}

\newtheorem*{conj*}{Semisimplicity  Conjecture}

\theoremstyle{remark}

\newtheorem*{Acknowledgements*}{Acknowledgements}

\theoremstyle{definition}

\newtheorem{eg}[thm]{Example}

\newcommand{\SStd}{{\rm SStd}}
\newcommand{\PStd}{{\rm PStd}}

\newcommand{\la}{\lambda}

 \newcommand{\SSTS}{\mathsf{S}}
 \newcommand{\SSTt}{\mathsf{t}}

\newcommand{\SSTT}{\mathsf{T}}  

\newcommand{\NN}{{\mathbb N}}

\let\<=\langle
\let\>=\rangle

\tikzset{
    ultra thin/.style= {line width=0.05pt},
    very thin/.style=  {line width=0.2pt},
    thin/.style=       {line width=0.1pt},
    semithick/.style=  {line width=0.6pt},
    thick/.style=      {line width=0.8pt},
    very thick/.style= {line width=1.2pt},
    ultra thick/.style={line width=1.6pt}
}

\crefname{defn}{Definition}{Definitions}
\crefname{thm}{Theorem}{Theorems}
\crefname{prop}{Proposition}{Propositions}
\crefname{lem}{Lemma}{Lemmas}
\crefname{cor}{Corollary}{Corollaries}
\crefname{conj}{Conjecture}{Conjectures}
\crefname{section}{Section}{Sections}
\crefname{subsection}{Subsection}{Subsections}
\crefname{eg}{Example}{Examples}
\crefname{figure}{Figure}{Figures}
\crefname{rem}{Remark}{Remarks}
\crefname{rmk}{Remark}{Remarks}
\crefname{equation}{equation}{equation}

\Crefname{defn}{Definition}{Definitions}
\Crefname{thm}{Theorem}{Theorems}
\Crefname{prop}{Proposition}{Propositions}
\Crefname{lem}{Lemma}{Lemmas}
\Crefname{cor}{Corollary}{Corollaries}
\Crefname{conj}{Conjecture}{Conjectures}
\Crefname{section}{Section}{Sections}
\Crefname{subsection}{Subsection}{Subsections}
\Crefname{eg}{Example}{Examples}
\Crefname{figure}{Figure}{Figures}
\Crefname{rem}{Remark}{Remarks}
\Crefname{rmk}{Remark}{Remarks}

 \newlength{\mylen}
\setbox1=\hbox{$\bullet$}\setbox2=\hbox{\tiny$\bullet$}
\usepackage{amsmath}
\newcommand\Item[1][]{%
  \ifx\relax#1\relax  \item \else \item[#1] \fi
  \abovedisplayskip=0pt\abovedisplayshortskip=0pt~\vspace*{-\baselineskip}}

\def\Item{\item\abovedisplayskip=0pt\abovedisplayshortskip=5pt~\vspace*{-\baselineskip}}

\begin{document}

\title[The uniqueness of  plethystic  factorisation]{The uniqueness of   plethystic  factorisation} 
\author{Chris Bowman}
\author{Rowena Paget}
 
\maketitle
 
\section*{Introduction}
Let $s_\lambda$ and $s_\mu$ denote the Schur functions labelled by the partitions $\la$ and $\mu$.  
There are three ways of ``multiplying" this pair of functions together in order to obtain a new symmetric function; these are the   Littlewood--Richardson, Kronecker, and plethysm products. 
 The primary purpose of this  paper is to address the most fundamental question one can ask of  such a  product: {\em ``does  
it  factorise uniquely?"}.    
For the Littlewood--Richardson product, this question  was answered by  Rajan  \cite{Rajan}.  
We solve this question  for   the most difficult  and mysterious 
    of    these products, the plethysm product (which we denote $\circ$) as follows.
 \begin{thmA*}
 Let $\mu,\nu,  \pi,\rho$ be arbitrary partitions.  
   If 
 $ 
 s_\nu\circ s_\mu = s_\rho \circ s_\pi
  $
 then either $\nu=\rho$ and $\mu=\pi$;  
  or we are  in one of five exceptional cases, 
$$\begin{array}{llll}
& s_{(2,1^2)}\circ s_{(1)}= s_{(1^2)}\circ s_{(1^2)},
\quad\quad\quad
& s_{(3,1)}\circ s_{(1)}= s_{(1^2)}\circ s_{(2)},
\\
 &s_{(2,1^2)}\circ s_{(2)}  = 
s_{(1^2)}\circ s_{(3,1)},   &  
s_{(2,1^2)}\circ s_{(1^2)}  = 
s_{(1^2)}\circ s_{(2,1^2)}, \\
&s_\nu \circ s_{(1)}=s_{(1)}\circ s_\nu.
\end{array}$$
   \end{thmA*}

 In general, the decomposition of a plethysm product will have very, very many constituents. 
We  ask: {\em``when is the plethysm product of two Schur functions indecomposable?"}. 
We prove that in fact such a product is always decomposable, and even inhomogeneous, 
except for some obvious exceptions.  The analogous result  for the Kronecker product was obtained by   Bessenrodt and Kleshchev  \cite{BK}.

\begin{thmB*}
Let $\mu,\nu $ be  partitions.  
The product 
  $
 s_\nu\circ s_\mu 
 $ 
 is decomposable and inhomogeneous except in the following exceptional cases:
 $$
 s_{(1^2)}\circ s_{(1^2)}= s_{(2,1^2)},
\quad
 s_{(1^2)}\circ s_{(2)}= s_{(3,1)},
 \quad 
  s_\nu \circ s_{(1)}=s_\nu=
   s_{(1)}\circ    s_\nu  . $$

\end{thmB*}

 Understanding and decomposing the  Kronecker and   plethystic products of pairs of Schur functions was identified by Richard Stanley as two of the most important open problems in   algebraic combinatorics \cite[Problems 9 \& 10]{Sta00}. 
Almost nothing is known about general constituents of plethysm products;  however the maximal
 terms in the dominance ordering are now well-understood   \cite{PW}.
 Our proof of    Theorems A and B proceeds  by careful analysis  of these maximal terms.   
   \color{black}

Outside of combinatorics,  
 plethysm products arise naturally  in the representation theory of symmetric and general linear groups.  
In quantum information  theory, the positivity of  
constituents  in a  plethysm product of two Schur functions      is equivalent to the existence of  
quantum states with certain spectra, margins, and occupation numbers \cite{MR2421478,MR2745569}.  
Decomposing     Kronecker  and  plethystic  products  of Schur functions is  the  central plank of Geometric Complexity Theory, an approach that seeks to settle the P versus NP problem \cite{MR1861288};  this approach was recently shown to require not only knowledge of the positivity but also precise information on the actual multiplicities of the  constituents  of the products $s_\nu \circ s_\mu$ \cite{MR3868002}.

 \section{Partitions, symmetric functions\\  and maximal terms in plethysm} 
 
  We define a {\sf composition} $\lambda\vDash n$ to be a   finite  sequence  of non-negative integers $ (\lambda_1,\lambda_2, \ldots)$ whose sum, $|\lambda| = \lambda_1+\lambda_2 + \dots$, equals $n$.
  If the sequence $(\lambda_1,\lambda_2, \ldots)$ is  weakly decreasing,  we say that $\la$ is a {\sf partition} and write $\la\vdash n$.       Given  $\lambda  $ a partition of $n$, the {\sf Young diagram}   is defined to be the configuration of nodes
\[
[\la]=\{(r,c) \mid  1\leq  c\leq \lambda_r\}.
\]
We say that a partition is {\sf linear} if it consists only of one row, or one column.  
 The conjugate partition, $\la^T$, is the partition obtained by interchanging the rows and columns of $\la$.  The number of non-zero parts of a partition, $\la$,  is called its {\sf length}, $\ell(\lambda)$; the size of the largest part is called the {\sf width}, $w(\la)$; the sum of all the parts of $\lambda$ is called its  size.

 Given two partitions $\la $ and $\mu$, we let  $\la + \mu$  and  $\la \sqcup \mu$ denote the partitions obtained by adding the partition horizontally and vertically respectively.    In more detail
 $$
 \la+\mu=(\la_1+\mu_1,\la_2+\mu_2,\la_3+\mu_3,\dots )
 $$
 and $\la\sqcup\mu$ is the partition whose multiset of parts is the disjoint union of the multisets of parts of $\la$ and $\mu$. We have that  
  $$\la\sqcup\mu= (\la^T+\mu^T)^T.$$
Finally we remark that, in this paper, the partition $\la\sqcup\mu$ is usually equal to $$(\la_1,\la_2,\dots,\la_{\ell(\la)},\mu_1,\mu_2,\dots,\mu_{\ell(\mu)}). $$ In other words, we often do not need to reorder the multisets of parts --- this is simply because $\la_{\ell(\la)}\ge \mu_1$ in most cases.  

 We   now recall the {\sf dominance ordering} on partitions. Let $\la,\mu$ be partitions. 
 We  write  
 $\la\trianglerighteq \mu$  if 
 $$
 \sum_{1\leq i \leq k}\la_i \geq  \sum_{1\leq i \leq k}\mu_i \text{  for all } k\geq 1.
 $$
 If $\la\trianglerighteq \mu$ and $\la\neq \mu$ we write $\la\rhd \mu$.  The dominance ordering  is a partial ordering on the set of  partitions of a given size.  This partial order can be refined into a total ordering as follows: 
  we write 
$\la\succ \mu$  if 
 $$
\text{  $\la_k >\mu_k$ for some $k\geq 1$ and  }
 \la_i = \mu_i \text{ for all } 1\leq i \leq k-1.
 $$
  We refer to   $\succ $  as the {\sf lexicographic ordering}. 
    We now define the  {\sf transpose-lexicographic} ordering  as follows:
    $$
       \la\succ_T\mu \text{ if and only if }
      \la^T\succ \mu^T .    
    $$
  We emphasise that this total ordering is not simply the opposite ordering to the lexicographic ordering;  minimality with respect to $\succ $ is not equivalent to maximality with respect to $\succ_T$.  
  
   Let   $\lambda   $ be  a partition  of $n$. 
 A {\sf    Young tableau of shape $  \lambda$} may be defined as  a map
 $\SSTt : [\la] \to \NN.$  Recall that the tableau $\SSTt$ is {\sf semistandard} if  $\SSTt(r,c-1)\leq \SSTt(r,c)$   and
$\SSTt(r-1,c)< \SSTt(r,c)$  for all $(r,c)\in [\la]$.  
 We let $\SSTt_k = |\{ (r,c)\in [\la] \mid \SSTt(r,c)=k\}|$ for $k\in \NN$.  
 We refer to the composition $\alpha=(\SSTt_1,\SSTt_2,\SSTt_3,\dots)$ as the {\sf weight} of the tableau~$\SSTt$.  
 We denote the set of all  tableaux of shape $\la$ by $\SStd_\NN(\la)$,  and the subset of those having weight $\alpha$ by $\SStd_\NN(\la,\alpha)$.    
The {\sf Schur function} $s_\la$, for $\la$ a partition of $n$, may be defined as follows:
   $$
   s_\la = \sum_{
   \begin{subarray}c
   \alpha \vDash n
   \end{subarray}
   }  
| \SStd_\NN(\la,\alpha)
| x^\alpha
 \qquad  
\text{   where} \qquad   x^\alpha= x_1^{\alpha_1} x_2^{\alpha_2} x_3^{\alpha_3}\dots . $$

The  {\sf plethysm product} of two symmetric functions 
 is defined in \cite[Chapter~7,   A2.6]{Sta99} or   \cite[Chapter I.8]{MR3443860}. 
 The plethysm product of two Schur functions  is again a symmetric function  and so can be rewritten as a linear combination of Schur functions:
$$ 
s_{\nu}\circ s_{\mu} 
 =\sum_\alpha p(\nu,\mu,\alpha) s_{\alpha}
 $$
  such that $p(\nu,\mu,\alpha) \geq 0$.   
We say that the product is {\sf homogeneous} if there is precisely one partition, $\alpha$, such that  $p(\nu,\mu,\alpha) > 0$; we say that the product is 
 {\sf indecomposable} if, in addition,   $p(\nu,\mu,\alpha) =1$.  
We now recall the role  conjugation -- often called the $\omega$ involution -- plays   in plethysm (see, for example, \cite[Ex. 1, Chapter I.8]{MR3443860}). For $\mu\vdash m$, $\nu \vdash n$, and $\alpha\vdash mn$  we have that 
\begin{equation}\label{conjugate}
p( \nu,\mu,\alpha)=
\begin{cases}
p( \nu,\mu^T,\alpha^T)				&\text{if $m$ is even}\\
p( \nu^T,\mu^T,\alpha^T)				&\text{if $m$ is odd.}  
\end{cases}
\end{equation}
Throughout this paper we shall let $\mu,\nu,  \pi,\rho$ be partitions of $m, n, p$ and $q$ respectively.  In order to keep track of the effect of this conjugation 
when comparing products 
$s_\nu \circ s_\mu $ and $s_\rho \circ s_\pi $,  
 we set 
$$ \nu^M=\begin{cases}
\nu 		&\text{if $m$ is even}\\
\nu^T &\text{if $m$ is odd} 
\end{cases}
\qquad
 \rho^P=\begin{cases}
\rho 		&\text{if $p$ is even}\\
\rho^T &\text{if $p$ is odd} 
\end{cases}
$$
 Given   a total ordering, $>$, on partitions we let   $${ \rm max}_{> } (s_{\nu}\circ s_{\mu})$$
 denote the  unique  partition, $\lambda$,  such that $p(\nu,\mu,\lambda)\neq 0$ and  $p(\nu,\mu,\alpha)=0$ for all $\alpha >  \lambda$.  We shall use this with both the lexicographic $\succ$ and transpose-lexicographic $\succ_T$ orderings .   By \cref{conjugate} we have that
$$  { \rm max}_{\succ_T } (s_{\nu } \circ s_{\mu })  
 =  ({ \rm max}_{\succ } (s_{\nu^M} \circ s_{\mu^T}))^T   .$$
 
 \begin{figure}[ht!]
$$\scalefont{0.6} 
 \begin{tikzpicture} [scale=0.5]
\clip(-0.1,5) rectangle (18.1,-11);
\path(0,0) coordinate (origin); 
\draw(origin)--++(0:18)--++(-90:1)--++(180:18)--++(90:1);
  \node  at (9,-0.5) {$n\mu_1 $};

\path(0,-1) coordinate (origin);

\draw(origin)--++(0:14)--++(-90:1)--++(180:14)--++(90:1);
  \node  at (7,-1.5) {$n\mu_2 $};
\draw[densely dotted] (14,-2)--++(-90:1)--++(180:4)--++(-90:1);
\draw[densely dotted] (0,-2)--++(-90:2);

\path(0,-4) coordinate (origin);

\draw(origin)--++(0:10)--++(-90:1)--++(180:10)--++(90:1);
  \node  at (5,-4.5) {$n\mu_{\ell-1} $};
  \path(0,-5) coordinate (origin); 
  \draw[fill=cyan!30](6,-5)--++(0:2)--++(-90:1)--++(180:2);
\draw(7,-5.5) node {$\nu_1$};  

 \draw[fill=cyan!30](0,-6)--++(0:2)--++(-90:1)--++(180:1)--++(-90:1)--++(180:1)--++(90:2);

\draw(1,-6.5) node {$\nu_{>1}$}; 
 
\draw(origin)--++(0:6)--++(-90:1)--++(180:6)--++(90:1);
  \node  at (3,-5.5) {$n\mu_{\ell}-n $};

\end{tikzpicture}
\qquad  \begin{tikzpicture} [scale=0.45]

\path(1,0) coordinate (origin); 

\draw(origin)--++(-90:18)--++(0:1)--++(90:18)--++(180:1);
   \node[rotate=-90] at (1.5,-9) {$n\mu_1^T$};

\path(2,0) coordinate (origin); 

\draw(origin)--++(-90:14)--++(0:1)--++(90:14)--++(180:1);
  \node[rotate=-90] at (2.5,-7) {$n\mu_2^T$};

\path(5,0) coordinate (origin); 

\draw(origin)--++(-90:10)--++(0:1)--++(90:10)--++(180:1);
  \node[rotate=-90] at (5.5,-5) {$n\mu_{k- {\scalefont{0.9}1}}^T$};

\path(6,0) coordinate (origin); 

\draw(origin)--++(-90:6)--++(0:1)--++(90:6)--++(180:1);
  \node[rotate=-90] at (6.5,-3) {$n\mu_{k }^T-n$};

 \draw[fill=cyan!30](6,-6)--++(-90:2)--++(0:1)--++(90:2)--++(180:1);
  \node[rotate=-90] at (6.5,-7) {$\nu_1^M$};
 
  \draw[fill=cyan!30](7,0)--++(-90:2)--++(0:1) --++(90:1) --++(0:1) --++(90:1) --++(180:3);
 
  \node[rotate=-90] at (7.5,-1) {$\nu_{>1}^M$};
 
 \draw[densely dotted](1,0)--(6,0);

     \draw[densely dotted](3,-14)--++(0:1)--++(90:4)--++(0:1) coordinate  (origin);

\end{tikzpicture}$$

\!\!\!\!
\caption{Examples of the partitions 
${ \rm max}_{\succ } (s_\nu \circ s_\mu)   
$ and  
$
{ \rm max}_{\succ_T } (s_\nu \circ s_\mu) $
 for $\mu\vdash m$ and $\nu\vdash n$ with $\ell(\mu)=\ell$ and $w(\mu)=k$.  
 }
\label{maxlex}
\end{figure}

 The following theorems will be incredibly important in our arguments.

\begin{thm}[{\cite[Corollary 9.1]{PW}} and \cite{arxiv}]
\label{pppppppppp}
\color{black}Let $\mu $,  $\nu $ be partitions of $m$ and $n$ respectively.  
The unique maximal terms of 
$s_\nu \circ s_\mu$
 in the lexicographic and transpose lexicographic ordering are as follows :
  $$
{ \rm max}_{\succ } (s_\nu \circ s_\mu)  =
(n\mu_1,n\mu_2,\dots, n\mu_{\ell(\mu) -1},n\mu_{\ell(\mu)}-n+\nu_1, \nu_2, \dots ,\nu_{\ell(\nu)}) ,
 $$
  $$
{ \rm max}_{\succ_T } (s_\nu \circ s_\mu)  =
 (n\mu_1^T,n\mu_2^T,\dots, n\mu^T_{\mu_1 -1},n\mu^T_{\mu_1}-n+\nu_1^M, \nu_2^M, \dots ,\nu_{\ell(\nu^M)}^M) )^T.
 $$
 Moreover, we have that  $$p(\nu,\mu,{ \rm max}_{\succ } (s_\nu \circ s_\mu) )=1=
 p(\nu,\mu, { \rm max}_{\succ_T } (s_{\nu } \circ s_{\mu })   ).$$  
\end{thm}

 \begin{eg} \label{PWex}
When $\mu=(m)$, Theorem~\ref{PW} shows that
  $$ { \rm max}_{\succ } (s_\nu \circ s_{(m)}) = (nm-n) + \nu, \quad 
{ \rm max}_{\succ_T } (s_\nu \circ s_{(m)})  =
 ((n^{m-1}) \sqcup  \nu^M)^T.
 $$
\end{eg}

Sometimes we shall use the dominance ordering  $\rhd $ to compare the summands of $s_{\nu}\circ s_{\mu}$, and then there will, in general, be many  (incomparable) maximal partitions. To understand these summands, we require some further definitions.    We place a lexicographic  ordering, $\prec$, on the set of semistandard Young tableaux as follows.
  Let $\SSTS\neq\SSTT$ be semistandard $\mu$-tableaux,   and consider the leftmost column in which $\SSTS$ and $\SSTT$ differ.  We write 
 $\SSTS \prec \SSTT$ if  the greatest entry not appearing in both columns lies in $\SSTT$.  Following \cite[Definition 1.4]{deBPW}, we define a {\sf    plethystic  tableau of shape $  \mu^\nu$}  and weight $\alpha$ to be  a map
 $$\SSTT : [\nu] \to \SStd_\NN(\mu)$$ such that the total  number of occurrences of
$k$ in the tableau entries of $\SSTT$ is $\alpha_k$ for each $k$.
 We say that such a tableau is {\sf semistandard} if  $\SSTT(r,c-1) \preceq  \SSTT(r,c)$   and
$\SSTT(r-1,c)  \prec  \SSTT(r,c)$  for all $(r,c)\in [\nu]$.  
   We denote the set of all plethystic  tableaux of shape $\mu^\nu$ and weight $\alpha$ by by $\PStd(\mu^\nu,
  \alpha)$.

 \begin{figure}[ht!]
$$  \begin{tikzpicture} [scale=0.6]

\path(0,0) coordinate (origin); 
 \begin{scope}{\draw[very thick](origin)--++(0:2)--++(-90:1)--++(180:1)--++(-90:1)--++(180:1)--++(90:2); 

   \draw(0.5,-0.5)  node {1};
       \draw(1.5,-0.5)node {1};
              \draw(0.5,-1.5)node {2};
       
  \clip (origin)--++(0:2)--++(-90:1)--++(180:1)--++(-90:1)--++(180:1)--++(90:2);

  \path(origin) coordinate (origin);
   \foreach \i in {1,...,19}
  {
    \path (origin)++(0:1*\i cm)  coordinate (a\i);
    \path (origin)++(-90:1*\i cm)  coordinate (b\i);
    \path (a\i)++(-90:10cm) coordinate (ca\i);
    \path (b\i)++(0:10cm) coordinate (cb\i);
    \draw[thin] (a\i) -- (ca\i)  (b\i) -- (cb\i); }

    } \end{scope}

  \path(3,0) coordinate (origin); 
 \begin{scope}{\draw[very thick](origin)--++(0:2)--++(-90:1)--++(180:1)--++(-90:1)--++(180:1)--++(90:2); 
      \path(origin)--++(0:0.5)--++(-90:0.5)  node {1};
       \path(origin)--++(0:1.5)--++(-90:0.5)node {1};
            \path(origin)--++(0:0.5)--++(90:-1.5)node {3};

  \clip (origin)--++(0:1)--++(-90:1)--++(180:1)--++(-90:1)--++(180:1)--++(90:1); 
  \path(origin) coordinate (origin);
   \foreach \i in {1,...,19}
  {
    \path (origin)++(0:1*\i cm)  coordinate (a\i);
    \path (origin)++(-90:1*\i cm)  coordinate (b\i);
    \path (a\i)++(-90:10cm) coordinate (ca\i);
    \path (b\i)++(0:10cm) coordinate (cb\i);
    \draw[thin] (a\i) -- (ca\i)  (b\i) -- (cb\i); }

 \path(0.5,-0.5) coordinate (origin);
 \foreach \i in {1,...,19}
  {
    \path (origin)++(0:1*\i cm)  coordinate (a\i);
    \path (origin)++(-90:1*\i cm)  coordinate (b\i);
    \path (a\i)++(-90:1cm) coordinate (ca\i);
        \path (ca\i)++(-90:1cm) coordinate (cca\i);
    \path (b\i)++(0:1cm) coordinate (cb\i);
    \path (cb\i)++(0:1cm) coordinate (ccb\i);
  }} \end{scope}

    \path(6,0) coordinate (origin); 
 \begin{scope}{\draw[very thick](origin)--++(0:2)--++(-90:1)--++(180:1)--++(-90:1)--++(180:1)--++(90:2); 
         \path(origin)--++(0:0.5)--++(-90:0.5)  node {1};
       \path(origin)--++(0:1.5)--++(-90:0.5)node {1};
            \path(origin)--++(0:0.5)--++(90:-1.5)node {3}; 

  \clip (origin)--++(0:1)--++(-90:1)--++(180:1)--++(-90:1)--++(180:1)--++(90:1); 
  \path(origin) coordinate (origin);
   \foreach \i in {1,...,19}
  {
    \path (origin)++(0:1*\i cm)  coordinate (a\i);
    \path (origin)++(-90:1*\i cm)  coordinate (b\i);
    \path (a\i)++(-90:10cm) coordinate (ca\i);
    \path (b\i)++(0:10cm) coordinate (cb\i);
    \draw[thin] (a\i) -- (ca\i)  (b\i) -- (cb\i); }

 \path(0.5,-0.5) coordinate (origin);
 \foreach \i in {1,...,19}
  {
    \path (origin)++(0:1*\i cm)  coordinate (a\i);
    \path (origin)++(-90:1*\i cm)  coordinate (b\i);
    \path (a\i)++(-90:1cm) coordinate (ca\i);
        \path (ca\i)++(-90:1cm) coordinate (cca\i);
    \path (b\i)++(0:1cm) coordinate (cb\i);
    \path (cb\i)++(0:1cm) coordinate (ccb\i);
  }} \end{scope}

  \path(0,-3) coordinate (origin); 
     \path(origin)--++(0:0.5)--++(-90:0.5)  node {1};
       \path(origin)--++(0:1.5)--++(-90:0.5)node {2};
            \path(origin)--++(0:0.5)--++(90:-1.5)node {3}; 
 \begin{scope}{\draw[very thick](origin)--++(0:2)--++(-90:1)--++(180:1)--++(-90:1)--++(180:1)--++(90:2); 
    
  \clip (origin)--++(0:1)--++(-90:1)--++(180:1)--++(-90:1)--++(180:1)--++(90:1); 
  \path(origin) coordinate (origin);
   \foreach \i in {1,...,19}
  {
    \path (origin)++(0:1*\i cm)  coordinate (a\i);
    \path (origin)++(-90:1*\i cm)  coordinate (b\i);
    \path (a\i)++(-90:10cm) coordinate (ca\i);
    \path (b\i)++(0:10cm) coordinate (cb\i);
    \draw[thin] (a\i) -- (ca\i)  (b\i) -- (cb\i); }

 \path(0.5,-0.5) coordinate (origin);
 \foreach \i in {1,...,19}
  {
    \path (origin)++(0:1*\i cm)  coordinate (a\i);
    \path (origin)++(-90:1*\i cm)  coordinate (b\i);
    \path (a\i)++(-90:1cm) coordinate (ca\i);
        \path (ca\i)++(-90:1cm) coordinate (cca\i);
    \path (b\i)++(0:1cm) coordinate (cb\i);
    \path (cb\i)++(0:1cm) coordinate (ccb\i);
  }} \end{scope}

  \path(3,-3) coordinate (origin); 
 \begin{scope}{\draw[very thick](origin)--++(0:2)--++(-90:1)--++(180:1)--++(-90:1)--++(180:1)--++(90:2); 
      \path(origin)--++(0:0.5)--++(-90:0.5)  node {1};
       \path(origin)--++(0:1.5)--++(-90:0.5)node {1};
            \path(origin)--++(0:0.5)--++(90:-1.5)node {4}; 
   
  \clip (origin)--++(0:1)--++(-90:1)--++(180:1)--++(-90:1)--++(180:1)--++(90:1); 
  \path(origin) coordinate (origin);
   \foreach \i in {1,...,19}
  {
    \path (origin)++(0:1*\i cm)  coordinate (a\i);
    \path (origin)++(-90:1*\i cm)  coordinate (b\i);
    \path (a\i)++(-90:10cm) coordinate (ca\i);
    \path (b\i)++(0:10cm) coordinate (cb\i);
    \draw[thin] (a\i) -- (ca\i)  (b\i) -- (cb\i); }

 \path(0.5,-0.5) coordinate (origin);
 \foreach \i in {1,...,19}
  {
    \path (origin)++(0:1*\i cm)  coordinate (a\i);
    \path (origin)++(-90:1*\i cm)  coordinate (b\i);
    \path (a\i)++(-90:1cm) coordinate (ca\i);
        \path (ca\i)++(-90:1cm) coordinate (cca\i);
    \path (b\i)++(0:1cm) coordinate (cb\i);
    \path (cb\i)++(0:1cm) coordinate (ccb\i);
  }} \end{scope}

 \draw[very thick] (-0.5,0.5)--(8.5,0.5)--(8.5,-2.5)--(5.5,-2.5)--(5.5,-5.5)--(-0.5,-5.5)--(-0.5,0.5);
 \draw[very thick] (5.5,-2.5)--(-0.5,-2.5);
  \draw[very thick] (5.5,-2.5)--(5.5,0.5);
    \draw[very thick] (2.5,-5.5)--(2.5,0.5);
   \end{tikzpicture} 
   \qquad     
    \begin{tikzpicture} [scale=0.6]

\path(0,0) coordinate (origin); 
 \begin{scope}{\draw[very thick](origin)--++(0:2)--++(-90:1)--++(180:1)--++(-90:1)--++(180:1)--++(90:2); 

   \draw(0.5,-0.5)  node {1};
       \draw(1.5,-0.5)node {1};
              \draw(0.5,-1.5)node {2};
       
  \clip (origin)--++(0:2)--++(-90:1)--++(180:1)--++(-90:1)--++(180:1)--++(90:2);

  \path(origin) coordinate (origin);
   \foreach \i in {1,...,19}
  {
    \path (origin)++(0:1*\i cm)  coordinate (a\i);
    \path (origin)++(-90:1*\i cm)  coordinate (b\i);
    \path (a\i)++(-90:10cm) coordinate (ca\i);
    \path (b\i)++(0:10cm) coordinate (cb\i);
    \draw[thin] (a\i) -- (ca\i)  (b\i) -- (cb\i); }

    } \end{scope}

  \path(3,0) coordinate (origin); 
 \begin{scope}{\draw[very thick](origin)--++(0:2)--++(-90:1)--++(180:1)--++(-90:1)--++(180:1)--++(90:2); 
      \path(origin)--++(0:0.5)--++(-90:0.5)  node {1};
       \path(origin)--++(0:1.5)--++(-90:0.5)node {1};
            \path(origin)--++(0:0.5)--++(90:-1.5)node {2};

  \clip (origin)--++(0:1)--++(-90:1)--++(180:1)--++(-90:1)--++(180:1)--++(90:1); 
  \path(origin) coordinate (origin);
   \foreach \i in {1,...,19}
  {
    \path (origin)++(0:1*\i cm)  coordinate (a\i);
    \path (origin)++(-90:1*\i cm)  coordinate (b\i);
    \path (a\i)++(-90:10cm) coordinate (ca\i);
    \path (b\i)++(0:10cm) coordinate (cb\i);
    \draw[thin] (a\i) -- (ca\i)  (b\i) -- (cb\i); }

 \path(0.5,-0.5) coordinate (origin);
 \foreach \i in {1,...,19}
  {
    \path (origin)++(0:1*\i cm)  coordinate (a\i);
    \path (origin)++(-90:1*\i cm)  coordinate (b\i);
    \path (a\i)++(-90:1cm) coordinate (ca\i);
        \path (ca\i)++(-90:1cm) coordinate (cca\i);
    \path (b\i)++(0:1cm) coordinate (cb\i);
    \path (cb\i)++(0:1cm) coordinate (ccb\i);
  }} \end{scope}

    \path(6,0) coordinate (origin); 
 \begin{scope}{\draw[very thick](origin)--++(0:2)--++(-90:1)--++(180:1)--++(-90:1)--++(180:1)--++(90:2); 
         \path(origin)--++(0:0.5)--++(-90:0.5)  node {1};
       \path(origin)--++(0:1.5)--++(-90:0.5)node {1};
            \path(origin)--++(0:0.5)--++(90:-1.5)node {2}; 

  \clip (origin)--++(0:1)--++(-90:1)--++(180:1)--++(-90:1)--++(180:1)--++(90:1); 
  \path(origin) coordinate (origin);
   \foreach \i in {1,...,19}
  {
    \path (origin)++(0:1*\i cm)  coordinate (a\i);
    \path (origin)++(-90:1*\i cm)  coordinate (b\i);
    \path (a\i)++(-90:10cm) coordinate (ca\i);
    \path (b\i)++(0:10cm) coordinate (cb\i);
    \draw[thin] (a\i) -- (ca\i)  (b\i) -- (cb\i); }

 \path(0.5,-0.5) coordinate (origin);
 \foreach \i in {1,...,19}
  {
    \path (origin)++(0:1*\i cm)  coordinate (a\i);
    \path (origin)++(-90:1*\i cm)  coordinate (b\i);
    \path (a\i)++(-90:1cm) coordinate (ca\i);
        \path (ca\i)++(-90:1cm) coordinate (cca\i);
    \path (b\i)++(0:1cm) coordinate (cb\i);
    \path (cb\i)++(0:1cm) coordinate (ccb\i);
  }} \end{scope}

  \path(0,-3) coordinate (origin); 
     \path(origin)--++(0:0.5)--++(-90:0.5)  node {1};
       \path(origin)--++(0:1.5)--++(-90:0.5)node {2};
            \path(origin)--++(0:0.5)--++(90:-1.5)node {2}; 
 \begin{scope}{\draw[very thick](origin)--++(0:2)--++(-90:1)--++(180:1)--++(-90:1)--++(180:1)--++(90:2); 
    
  \clip (origin)--++(0:1)--++(-90:1)--++(180:1)--++(-90:1)--++(180:1)--++(90:1); 
  \path(origin) coordinate (origin);
   \foreach \i in {1,...,19}
  {
    \path (origin)++(0:1*\i cm)  coordinate (a\i);
    \path (origin)++(-90:1*\i cm)  coordinate (b\i);
    \path (a\i)++(-90:10cm) coordinate (ca\i);
    \path (b\i)++(0:10cm) coordinate (cb\i);
    \draw[thin] (a\i) -- (ca\i)  (b\i) -- (cb\i); }

 \path(0.5,-0.5) coordinate (origin);
 \foreach \i in {1,...,19}
  {
    \path (origin)++(0:1*\i cm)  coordinate (a\i);
    \path (origin)++(-90:1*\i cm)  coordinate (b\i);
    \path (a\i)++(-90:1cm) coordinate (ca\i);
        \path (ca\i)++(-90:1cm) coordinate (cca\i);
    \path (b\i)++(0:1cm) coordinate (cb\i);
    \path (cb\i)++(0:1cm) coordinate (ccb\i);
  }} \end{scope}

  \path(3,-3) coordinate (origin); 
 \begin{scope}{\draw[very thick](origin)--++(0:2)--++(-90:1)--++(180:1)--++(-90:1)--++(180:1)--++(90:2); 
      \path(origin)--++(0:0.5)--++(-90:0.5)  node {1};
       \path(origin)--++(0:1.5)--++(-90:0.5)node {1};
            \path(origin)--++(0:0.5)--++(90:-1.5)node {3}; 
   
  \clip (origin)--++(0:1)--++(-90:1)--++(180:1)--++(-90:1)--++(180:1)--++(90:1); 
  \path(origin) coordinate (origin);
   \foreach \i in {1,...,19}
  {
    \path (origin)++(0:1*\i cm)  coordinate (a\i);
    \path (origin)++(-90:1*\i cm)  coordinate (b\i);
    \path (a\i)++(-90:10cm) coordinate (ca\i);
    \path (b\i)++(0:10cm) coordinate (cb\i);
    \draw[thin] (a\i) -- (ca\i)  (b\i) -- (cb\i); }

 \path(0.5,-0.5) coordinate (origin);
 \foreach \i in {1,...,19}
  {
    \path (origin)++(0:1*\i cm)  coordinate (a\i);
    \path (origin)++(-90:1*\i cm)  coordinate (b\i);
    \path (a\i)++(-90:1cm) coordinate (ca\i);
        \path (ca\i)++(-90:1cm) coordinate (cca\i);
    \path (b\i)++(0:1cm) coordinate (cb\i);
    \path (cb\i)++(0:1cm) coordinate (ccb\i);
  }} \end{scope}

 \draw[very thick] (-0.5,0.5)--(8.5,0.5)--(8.5,-2.5)--(5.5,-2.5)--(5.5,-5.5)--(-0.5,-5.5)--(-0.5,0.5);
 \draw[very thick] (5.5,-2.5)--(-0.5,-2.5);
  \draw[very thick] (5.5,-2.5)--(5.5,0.5);
    \draw[very thick] (2.5,-5.5)--(2.5,0.5);
   \end{tikzpicture}   $$
   \caption{Two plethystic tableaux of shape $ {(2,1)}^{(3,2)}$.  
   The former has weight $(9,2,3,1)$ and the latter has weight $(9,5,1)$. 
   The latter is maximal in the dominance ordering; the former is not.   }
   \label{pleth-tab}
\end{figure}

\begin{thm}[{\cite[Theorem 1.5]{deBPW}}]\label{PW}The maximal partitions $\alpha$ in the dominance order such that $s_\alpha$ is a constituent of $s_\nu \circ s_\mu$ are precisely the maximal weights of the plethystic semistandard tableaux of shape $\mu^\nu$. Moreover if $\alpha$ is such a maximal partition then $p( \nu,\mu,\alpha)$ is equal to $|\PStd(\mu^\nu,
  \alpha)|$.    \end{thm}

Finally, we recall the one known case in which every term in a plethystic product is both maximal and minimal in the dominance ordering.
Given $\alpha$ a partition of $n$ with distinct parts, we let  
$2[\alpha]$ denote the unique partition of $2n$  whose leading diagonal hook-lengths  are $2\alpha_1, \dots, 2\alpha_{\ell(\alpha)}$  and whose $i\textsuperscript{th}$ row has length $\alpha_i+i$ for  $1\leq i \leq \ell(\alpha)$.
(An example follows.) We have the decomposition
\begin{equation}\label{maxminall}
s_{(1^n)}\circ s_{(2)}=\sum_{\alpha}  s_{2[\alpha]},
\end{equation}
where the sum is over all partitions $\alpha $  of $n$ into distinct parts.    
 This decomposition is given in \cite[Corollary 8.6]{MR3483115} and  \cite[I. 8, Exercise 6(d)]{MR3443860}.  We observe that for $n>2$ this product is never homogeneous (for example $\alpha=(n)$ and $\alpha=(n-1,1)$ both label summands).  

\begin{eg}
For $n=5$ the decomposition obtained is
$$
s_{(1^5)}\circ s_{(2)}=
s_{2[(3,2)]}+ s_{2[(4,1)]}+s_{2[(5)]} 
= s_{(4^2,2)}+s_{(5,3,1^2)}+s_{(6,1^4)}.
$$
We picture these partitions (and the manner in which they are formed) in \cref{2alpha} below.  We remark that 
$$
s_{(1^5)}\circ s_{(1^2)}=
  s_{(4^2,2)^T}+s_{(5,3,1^2)^T}+s_{(6,1^4)^T}
=s_{(3^2,2^2) }+s_{(4,2^2,1^2) }+s_{(5,1^5) }
$$
by \cref{conjugate} simply because $m=2$ is even.  
\end{eg} 
\color{black}

\begin{figure}[ht!]
$$\begin{minipage}{3cm} \begin{tikzpicture} [scale=0.5]

\fill[cyan!30](1,0)--++(0:3) --++(-90:2)--++(180:2)--++(90:1)--++(180:1);

\fill[white!30](0,0)--++(-90:3)--++(0:2) --++(90:2) --++(180:1) --++(90:1) ;  

 \draw[very thick](0,0)--++(0:4)--++(-90:1) --++(180:3) --++(90:1)  ;
 \draw[very thick](0,0)--++(-90:3)--++(0:1) --++(90:2) ;
 \draw[very thick](0,0)--++(-90:3)--++(0:2) --++(90:2) --++(-90:1) --++(0:2) --++(90:1) ;

  \clip(0,0)--++(0:4)--++(-90:2)--++(180:2)--++(-90:1)
 --++(180:2)--++(90:3)
  ;

  \path(0,0) coordinate (origin);
   \foreach \i in {1,...,19}
  {
    \path (origin)++(0:1*\i cm)  coordinate (a\i);
    \path (origin)++(-90:1*\i cm)  coordinate (b\i);
    \path (a\i)++(-90:10cm) coordinate (ca\i);
    \path (b\i)++(0:10cm) coordinate (cb\i);
    \draw[thin] (a\i) -- (ca\i)  (b\i) -- (cb\i); }

 \path(0.5,-0.5) coordinate (origin);
 \foreach \i in {1,...,19}
  {
    \path (origin)++(0:1*\i cm)  coordinate (a\i);
    \path (origin)++(-90:1*\i cm)  coordinate (b\i);
    \path (a\i)++(-90:1cm) coordinate (ca\i);
        \path (ca\i)++(-90:1cm) coordinate (cca\i);
    \path (b\i)++(0:1cm) coordinate (cb\i);
    \path (cb\i)++(0:1cm) coordinate (ccb\i);
  }

   \end{tikzpicture} \end{minipage}\qquad
   \begin{minipage}{3cm} \begin{tikzpicture} [scale=0.5]

  \draw[very thick,fill=white!30](0,0)--++(-90:4)--++(0:1)--++(90:2)--++(0:1)
  --++(90:1)--++(180:1)  --++(90:1); 
 \draw[very thick,fill=cyan!30](0,0)--++(0:5)--++(-90:1)--++(180:2)--++(-90:1)--++(180:1)--++(90:1)--++(180:1)--++(90:1);
 \draw[very thick](1,-1) rectangle (3,-2);

  \clip(0,0)--++(0:5)--++(-90:1)--++(180:2)--++(-90:1)
 --++(180:2) --++(-90:2) --++(180:1)--++(90:5)
  ;

  \path(0,0) coordinate (origin);
   \foreach \i in {1,...,19}
  {
    \path (origin)++(0:1*\i cm)  coordinate (a\i);
    \path (origin)++(-90:1*\i cm)  coordinate (b\i);
    \path (a\i)++(-90:10cm) coordinate (ca\i);
    \path (b\i)++(0:10cm) coordinate (cb\i);
    \draw[thin] (a\i) -- (ca\i)  (b\i) -- (cb\i); }

 \path(0.5,-0.5) coordinate (origin);
 \foreach \i in {1,...,19}
  {
    \path (origin)++(0:1*\i cm)  coordinate (a\i);
    \path (origin)++(-90:1*\i cm)  coordinate (b\i);
    \path (a\i)++(-90:1cm) coordinate (ca\i);
        \path (ca\i)++(-90:1cm) coordinate (cca\i);
    \path (b\i)++(0:1cm) coordinate (cb\i);
    \path (cb\i)++(0:1cm) coordinate (ccb\i);
  }

   \end{tikzpicture}  \end{minipage}
\qquad
   \begin{minipage}{3cm} \begin{tikzpicture} [scale=0.5]

  \draw[very thick]
  (0,0)--++(0:6)--++(-90:1)--++(180:5)--++(-90:4)
 --++(180:1) --++(90:5);
 
   \draw[very thick,fill=cyan!30]
  (1,0)--++(0:5)--++(-90:1)--++(180:5)--++(90:1)
;

  \draw[very thick,fill=white!30]
  (0,0)--++(-90:5)--++(0:1)--++(90:5)--++(180:1)
;

  \clip(0,0)--++(0:6)--++(-90:1)--++(180:5)--++(-90:4)
 --++(180:1) --++(90:5)   ;

  \path(0,0) coordinate (origin);
   \foreach \i in {1,...,19}
  {
    \path (origin)++(0:1*\i cm)  coordinate (a\i);
    \path (origin)++(-90:1*\i cm)  coordinate (b\i);
    \path (a\i)++(-90:10cm) coordinate (ca\i);
    \path (b\i)++(0:10cm) coordinate (cb\i);
    \draw[thin] (a\i) -- (ca\i)  (b\i) -- (cb\i); }

 \path(0.5,-0.5) coordinate (origin);
 \foreach \i in {1,...,19}
  {
    \path (origin)++(0:1*\i cm)  coordinate (a\i);
    \path (origin)++(-90:1*\i cm)  coordinate (b\i);
    \path (a\i)++(-90:1cm) coordinate (ca\i);
        \path (ca\i)++(-90:1cm) coordinate (cca\i);
    \path (b\i)++(0:1cm) coordinate (cb\i);
    \path (cb\i)++(0:1cm) coordinate (ccb\i);
  }

   \end{tikzpicture}  \end{minipage}$$
   \caption{The partitions $2[(3,2)]$,  $2[(4,1)] $ and ${2[(5)]} $ respectively.  }
   \label{2alpha}
\end{figure}
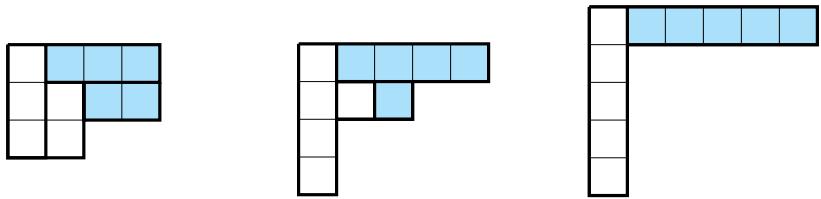

\section{Decomposability and homogeneity of plethysm }

In this section, we prove Theorem B of the introduction: namely we classify all decomposable/homogeneous   plethystic products of Schur functions.  
This also serves to remove the homogeneous products from consideration in the proof of Theorem~A.

\begin{thm}\label{homog}
Let $\mu,\nu $ be  partitions of $m$  and $n$, respectively.  
The product 
  $
 s_\nu\circ s_\mu 
 $ 
 is decomposable and inhomogeneous except in the following cases:
 $$
 s_{(1^2)}\circ s_{(1^2)}= s_{(2,1^2)},
\quad
 s_{(1^2)}\circ s_{(2)}= s_{(3,1)},
 \quad 
  s_\nu \circ s_{(1)}=s_\nu,
  \quad
 s_{(1)}\circ    s_\mu  =s_\mu.$$

\end{thm}
\begin{proof}
That the listed products are homogeneous is obvious.  We assume that $m, n \neq~1$ and
\begin{equation}\label{hellotheremr}
 {\rm max}_{\succ}(s_\nu\circ s_\mu ) ={\rm max}_{\succ_T}(s_\nu\circ s_\mu ).
 \end{equation}
We shall show that this implies   
 that $\nu =(1^2)$ and $\mu \vdash 2$.
  We first assume that $\mu $ is non-linear, that is $\mu$ is neither $(m)$ nor $(1^m)$. 
We set $k=\ell(\mu )$.  
 We draw a horizontal line across the Young diagrams of $ {\rm max}_{\succ}(s_\nu\circ s_\mu )$
  and $  ({\rm max}_{\succ}(s_{\nu^M}\circ s_{\mu^T} ))^T$  so that the partitions below each of these lines  each have strictly fewer than $n$ nodes in total and are maximal with respect to this property.  
For  $ {\rm max}_{\succ}(s_\nu\circ s_\mu )$, this line is drawn between the
 $k\textsuperscript{th}$   and $(k+1)\textsuperscript{th}$ rows (even though the $(k+1)\textsuperscript{th}$ row might be zero).   
  For $  ({\rm max}_{\succ}(s_{\nu^M}\circ s_{\mu^T} ))^T$, this line is drawn at some point after the $(n(k-1)+1)^\textsuperscript{th}$ row.  
Since $  k  < n(k-1)+1$ for $n>1$, we see that ${\rm max}_{\succ}(s_\nu\circ s_\mu )
\neq  ({\rm max}_{\succ}(s_{\nu^M}\circ s_{\mu^T} ))^T$ as required.  

   It remains to consider the case that $\mu$ is linear and we assume (by conjugating if necessary) that
  $\mu=(m)$. Then, as we saw in Example~\ref{PWex},
 $$
  {\rm max}_{\succ}(s_\nu\circ s_{(m)})=(mn-n)+ \nu, \quad 
({\rm max}_{\succ}(s_{\nu^M}\circ s_{(1^m)} ))^T = ( {(m-1)}^n)+(\nu^M)^T.
  $$
Therefore row~$n$ of $ {\rm max}_{\succ}(s_\nu\circ s_{(m)})$  has length $\nu_n$ which is at most 1, and  row~$n$ of 
$({\rm max}_{\succ}(s_{\nu^M}\circ s_{(1^m)} ))^T$ has length at least $m-1$. 
Since we are considering only $m \ge 2$, we conclude that $m=2$ and $\nu_n=1$, that is $\nu=(1^n)$.
From the closed formula for the decomposition of $s_{(1^n)}\circ s_{(2)}$  in \cref{maxminall}, and the resulting decomposition  of its plethystic conjugate $s_{(1^n)}\circ s_{(1^2)}$, we observe that the product is  homogeneous if and only if $n=1,2$. 
 \end{proof}

\begin{cor}
If  $s_\nu \circ s_{(1)}=s_\rho\circ s_\pi$ or $s_{(1)} \circ s_{\mu}=s_\rho\circ s_\pi$ then 
 either: $\pi=(1^2)$ and $\rho$ is a partition of~2; or  at least one of $\rho$ or $\pi$ has size~1.  
\end{cor}

Therefore in the remainder of the paper, we can and will assume that none of the indexing partitions in our plethystic products are equal to $(1)\vdash 1$.

\section{Unique factorisation of   plethysm  }

A quick scan of the diagrams in \cref{maxlex} tells us that the maximal terms in the  product under the lexicographic and transpose-lexicographic orderings encode a great deal of information concerning the multiplicands of the product. 
 We might even think that these maximal terms are enough to uniquely determine the multiplicands.  
 In fact, this is not the case (as the following example shows).

\begin{eg} Consider the plethysm products
 $$
s_{(3^3,2,1)} \circ s_{(1^2) }
\qquad\text{and}\quad
s_{(2,1)}\circ s_{(4,1^4)}.
$$
Both have the same maximal terms in the lexicographic  and transpose-lexicographic orderings, namely those labelled by $(12,3^3,2,1)$ and $(15,3^2,2,1)^T$. 
\cref{fig1,fig2} depict  how these two partitions can be seen to be maximal in the lexicographic and  transpose-lexicographic orderings using \cref{pppppppppp}.  
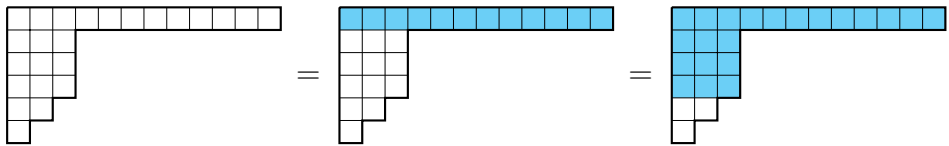
\begin{figure}[ht!]$$\begin{minipage}{3.7cm} \begin{tikzpicture} [scale=0.3]

  \draw[thick](0,0)--++(0:12)--++(-90:1)--++(180:9)--++(-90:3)
 --++(180:1)--++(-90:1)
 --++(180:1)--++(-90:1)
  --++(180:1)--++(90:6) ;
  \clip(0,0)--++(0:12)--++(-90:1)--++(180:9)--++(-90:3)
 --++(180:1)--++(-90:1)
 --++(180:1)--++(-90:1)
  --++(180:1)--++(90:6) ;  
  \path(0,0) coordinate (origin);
  
    \foreach \i in {1,...,19}
  {
    \path (origin)++(0:1*\i cm)  coordinate (a\i);
    \path (origin)++(-90:1*\i cm)  coordinate (b\i);
    \path (a\i)++(-90:10cm) coordinate (ca\i);
    \path (b\i)++(0:10cm) coordinate (cb\i);
    \draw[thin ] (a\i) -- (ca\i)  (b\i) -- (cb\i); }
 
   \path(0.5,-0.5) coordinate (origin);
 \foreach \i in {1,...,19}
  {
    \path (origin)++(0:1*\i cm)  coordinate (a\i);
    \path (origin)++(-90:1*\i cm)  coordinate (b\i);
    \path (a\i)++(-90:1cm) coordinate (ca\i);
        \path (ca\i)++(-90:1cm) coordinate (cca\i);
    \path (b\i)++(0:1cm) coordinate (cb\i);
    \path (cb\i)++(0:1cm) coordinate (ccb\i);
  } 

   \end{tikzpicture}
   \end{minipage}
=\; \begin{minipage}{3.7cm}
\begin{tikzpicture} [scale=0.3]
  \fill[cyan!50](0,0)--++(0:12)--++(-90:1)--++(180:12); 
  \draw[thick](0,0)--++(0:12)--++(-90:1)--++(180:9)--++(-90:3)
 --++(180:1)--++(-90:1)
 --++(180:1)--++(-90:1)
  --++(180:1)--++(90:6) ;
  \clip(0,0)--++(0:12)--++(-90:1)--++(180:9)--++(-90:3)
 --++(180:1)--++(-90:1)
 --++(180:1)--++(-90:1)
  --++(180:1)--++(90:6) ;  
  \path(0,0) coordinate (origin);

   \foreach \i in {1,...,19}
  {
    \path (origin)++(0:1*\i cm)  coordinate (a\i);
    \path (origin)++(-90:1*\i cm)  coordinate (b\i);
    \path (a\i)++(-90:10cm) coordinate (ca\i);
    \path (b\i)++(0:10cm) coordinate (cb\i);
    \draw[thin ] (a\i) -- (ca\i)  (b\i) -- (cb\i); }
 
   \path(0.5,-0.5) coordinate (origin);
 \foreach \i in {1,...,19}
  {
    \path (origin)++(0:1*\i cm)  coordinate (a\i);
    \path (origin)++(-90:1*\i cm)  coordinate (b\i);
    \path (a\i)++(-90:1cm) coordinate (ca\i);
        \path (ca\i)++(-90:1cm) coordinate (cca\i);
    \path (b\i)++(0:1cm) coordinate (cb\i);
    \path (cb\i)++(0:1cm) coordinate (ccb\i);
  } 

   \end{tikzpicture}
\end{minipage}= \; \begin{minipage}{3.7cm}
\begin{tikzpicture} [scale=0.3]
  \fill[cyan!50](0,0)--++(0:12)--++(-90:1)--++(180:9)--++(-90:3)
    --++(180:3)--++(90:4) ;
  \draw[thick](0,0)--++(0:12)--++(-90:1)--++(180:9)--++(-90:3)
 --++(180:1)--++(-90:1)
 --++(180:1)--++(-90:1)
  --++(180:1)--++(90:6) ;
  \clip(0,0)--++(0:12)--++(-90:1)--++(180:9)--++(-90:3)
 --++(180:1)--++(-90:1)
 --++(180:1)--++(-90:1)
  --++(180:1)--++(90:6) ;  
  \path(0,0) coordinate (origin);

   \foreach \i in {1,...,19}
  {
    \path (origin)++(0:1*\i cm)  coordinate (a\i);
    \path (origin)++(-90:1*\i cm)  coordinate (b\i);
    \path (a\i)++(-90:10cm) coordinate (ca\i);
    \path (b\i)++(0:10cm) coordinate (cb\i);
    \draw[thin ] (a\i) -- (ca\i)  (b\i) -- (cb\i); }
 
   \path(0.5,-0.5) coordinate (origin);
 \foreach \i in {1,...,19}
  {
    \path (origin)++(0:1*\i cm)  coordinate (a\i);
    \path (origin)++(-90:1*\i cm)  coordinate (b\i);
    \path (a\i)++(-90:1cm) coordinate (ca\i);
        \path (ca\i)++(-90:1cm) coordinate (cca\i);
    \path (b\i)++(0:1cm) coordinate (cb\i);
    \path (cb\i)++(0:1cm) coordinate (ccb\i);
  } 

   \end{tikzpicture}
\end{minipage}$$
\caption{Writing $(12,3^3,2,1)$ as $  {\rm max}_{\succ} (s_{(3^3,2,1)}\circ s_{(1^2)})$ and $ {\rm max}_{\succ} (s_{(2,1)} \circ s_{ (4,1^4)})$.}
\label{fig1}
\end{figure}

\begin{figure}[ht!]
$$\begin{minipage}{4.6cm} \begin{tikzpicture} [scale=0.3]

  \draw[thick](0,0)--++(0:15)--++(-90:1)--++(180:12)--++(-90:2)
 --++(180:1)--++(-90:1)
 --++(180:1)--++(-90:1)
  --++(180:1)--++(90:5) ;
  \clip(0,0)--++(0:15)--++(-90:1)--++(180:12)--++(-90:2)
 --++(180:1)--++(-90:1)
 --++(180:1)--++(-90:1)
  --++(180:1)--++(90:5) ;  \path(0,0) coordinate (origin);
  
    \foreach \i in {1,...,19}
  {
    \path (origin)++(0:1*\i cm)  coordinate (a\i);
    \path (origin)++(-90:1*\i cm)  coordinate (b\i);
    \path (a\i)++(-90:10cm) coordinate (ca\i);
    \path (b\i)++(0:10cm) coordinate (cb\i);
    \draw[thin ] (a\i) -- (ca\i)  (b\i) -- (cb\i); }
 
   \path(0.5,-0.5) coordinate (origin);
 \foreach \i in {1,...,19}
  {
    \path (origin)++(0:1*\i cm)  coordinate (a\i);
    \path (origin)++(-90:1*\i cm)  coordinate (b\i);
    \path (a\i)++(-90:1cm) coordinate (ca\i);
        \path (ca\i)++(-90:1cm) coordinate (cca\i);
    \path (b\i)++(0:1cm) coordinate (cb\i);
    \path (cb\i)++(0:1cm) coordinate (ccb\i);
  } 

   \end{tikzpicture}
   \end{minipage}
   = \;\begin{minipage}{4.6cm} \begin{tikzpicture} [scale=0.3]
  \fill[cyan!50](0,0)--++(0:12)--++(-90:1)--++(180:12);     
    
  \draw[thick](0,0)--++(0:15)--++(-90:1)--++(180:12)--++(-90:2)
 --++(180:1)--++(-90:1)
 --++(180:1)--++(-90:1)
  --++(180:1)--++(90:5) ;
  \clip(0,0)--++(0:15)--++(-90:1)--++(180:12)--++(-90:2)
 --++(180:1)--++(-90:1)
 --++(180:1)--++(-90:1)
  --++(180:1)--++(90:5) ;  \path(0,0) coordinate (origin);
  
    \foreach \i in {1,...,19}
  {
    \path (origin)++(0:1*\i cm)  coordinate (a\i);
    \path (origin)++(-90:1*\i cm)  coordinate (b\i);
    \path (a\i)++(-90:10cm) coordinate (ca\i);
    \path (b\i)++(0:10cm) coordinate (cb\i);
    \draw[thin ] (a\i) -- (ca\i)  (b\i) -- (cb\i); }
 
   \path(0.5,-0.5) coordinate (origin);
 \foreach \i in {1,...,19}
  {
    \path (origin)++(0:1*\i cm)  coordinate (a\i);
    \path (origin)++(-90:1*\i cm)  coordinate (b\i);
    \path (a\i)++(-90:1cm) coordinate (ca\i);
        \path (ca\i)++(-90:1cm) coordinate (cca\i);
    \path (b\i)++(0:1cm) coordinate (cb\i);
    \path (cb\i)++(0:1cm) coordinate (ccb\i);
  } 

   \end{tikzpicture}
   \end{minipage}  = \;\begin{minipage}{4.6cm} \begin{tikzpicture} [scale=0.3]
  \fill[cyan!50](0,0)--++(0:15)--++(-90:1)--++(180:12)--++(-90:2)--++(180:3);     
    
  \draw[thick](0,0)--++(0:15)--++(-90:1)--++(180:12)--++(-90:2)
 --++(180:1)--++(-90:1)
 --++(180:1)--++(-90:1)
  --++(180:1)--++(90:5) ;
  \clip(0,0)--++(0:15)--++(-90:1)--++(180:12)--++(-90:2)
 --++(180:1)--++(-90:1)
 --++(180:1)--++(-90:1)
  --++(180:1)--++(90:5) ;  \path(0,0) coordinate (origin);
  
    \foreach \i in {1,...,19}
  {
    \path (origin)++(0:1*\i cm)  coordinate (a\i);
    \path (origin)++(-90:1*\i cm)  coordinate (b\i);
    \path (a\i)++(-90:10cm) coordinate (ca\i);
    \path (b\i)++(0:10cm) coordinate (cb\i);
    \draw[thin ] (a\i) -- (ca\i)  (b\i) -- (cb\i); }
 
   \path(0.5,-0.5) coordinate (origin);
 \foreach \i in {1,...,19}
  {
    \path (origin)++(0:1*\i cm)  coordinate (a\i);
    \path (origin)++(-90:1*\i cm)  coordinate (b\i);
    \path (a\i)++(-90:1cm) coordinate (ca\i);
        \path (ca\i)++(-90:1cm) coordinate (cca\i);
    \path (b\i)++(0:1cm) coordinate (cb\i);
    \path (cb\i)++(0:1cm) coordinate (ccb\i);
  } 

   \end{tikzpicture}
   \end{minipage}
$$
\caption{Writing $(15,3^2,2,1)$ as
$ {\rm max}_{\succ} (s_{(3^3,2,1)}\circ s_{(2)})$ and $
 {\rm max}_{\succ} (s_{(2,1)} \circ s_{ (5,1^3)})$.}
\label{fig2}
\end{figure}
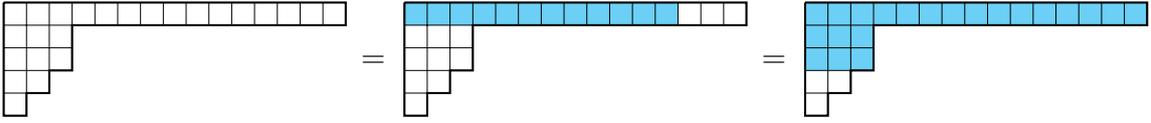

This puts a scupper on our plans to determine uniqueness solely using   maximal terms in the {\em lexicographic} and {\em transpose-lexicographic} orderings.  
 Now, we notice that the  plethysm products
 $s_{(3^3,2,1)}\circ s_{(1^2) }$
and $ s_{(2,1)}\circ s_{(4,1^4)}
$  can still  be distinguished by looking at the maximal terms for both products in  the   {\em dominance ordering}.  
For example, 
$   (11, 4, 4, 3, 2)$ labels a maximal term that appears in  $s_{(3^3,2,1)} \circ s_{(1^2) }$ but it is not a maximal term in  and  $s_{(2,1)}\circ s_{(4,1^4)}$. Similarly, $ (11, 4, 3, 3, 3)$   labels a maximal term in $s_{(2,1)}\circ s_{(4,1^4)}$ but not in $s_{(3^3,2,1)} \circ s_{(1^2) }$.
\end{eg}

Our method of proof will proceed to distinguish plethysm products by first using maximal terms in the lexicographic ordering and  only when necessary considering  the broader family of terms which are maximal in the dominance ordering. We first consider the case where $\mu$ consists of a single row. 
 
 \begin{thm} \label{Rowena}
 Let $\mu,\nu,  \pi,\rho$ be  partitions of $m,n,p,q>1$ respectively. 
 We suppose that   $\mu=(m)$.  
   If 
 $$
 s_\nu\circ s_\mu = s_\rho \circ s_\pi
 $$
 then either $\nu=\rho$ and $\mu=\pi$ or we are in the exceptional case 
 $$s_{(2,1^2)}\circ s_{(2)}  = 
s_{(1^2)}\circ s_{(3,1)}.  $$     

 \end{thm}
 \begin{proof}
From the set-up, we know $mn=pq$. We set $\ell(\pi)=c+1$ for some $c\geq 0$.  
 By assumption, we have that 
 \begin{align}\label{1}
 {\rm max}_{\succ}(s_\nu \circ s_{(m)})
 &=
  {\rm max}_{\succ}(s_\rho \circ s_\pi)\\ \label{2}
   {\rm max}_{\succ}(s_{\nu^M} \circ s_{(1^m)})
 &=
  {\rm max}_{\succ}(s_{\rho^P} \circ s_{\pi^T}).
 \end{align}
 As a warm-up, we first consider the case where $\pi$ is linear.  
 If $\mu=(m)$ and $\pi=(p)$  then (see Example~\ref{PWex}) \cref{2} says that
$(n^{m-1}) \sqcup (\nu^M)^T=(q^{p-1}) \sqcup (\rho^P)^T$. 
By comparing   widths  we deduce that      $q=n$. This   implies $m=p$ and then $\nu=\rho$.
  Now,  suppose that  $\mu=(m)$ and $\pi=(1^p)$.   
Then $ {\rm max}_{\succ}(s_\nu \circ s_{(m)})=(nm-n)+\nu$ which, as $m \geq 2$ and $\nu$ has size~$n$, has final column of  length~1. For   \cref{1} to hold,  the same to be true of  $ {\rm max}_{\succ}(s_\rho \circ s_(1^p))=(q^{p-1}) \sqcup \rho$; this implies $p=2$. Similarly, comparing the final columns of ${\rm max}_{\succ}(s_{\nu^M} \circ s_{(1^m)}) = (n^{m-1}) \sqcup \nu^M$ and $  {\rm max}_{\succ}(s_{\rho^P} \circ s_{(p)}) = (qp-q)+\rho^P$ also shows that $m=2$. Hence $n=q$ and we obtain a contradiction from comparing the widths of 
$(n) \sqcup \nu^M$ and $(q)+\rho^M$.

We now assume that $\pi$ is non-linear so   $\pi_1>1$ and $c>0$.  By \cref{2}, 
\begin{equation}\label{bigbox}
\color{black}(n^{m-1}) \sqcup  \nu^M 
\color{black}= (q\pi_1^T,q\pi_2^T,\dots, q\pi_{\pi_1-1}^T,q\pi_{\pi_1}^T-q+\rho_1^M,\rho_2^M,\dots  ).
\end{equation}
 Since $m\geq 2$   and $\pi_1>1$, it follows that $n=q\pi_1^T=q(c+1)$   and, as $mn=pq$, $p=(1+c)m$.  
\color{black}
   If $\nu^M=(n)$ then the left hand side of \cref{bigbox} is $(n^m)$.   
  Since  $q<n$, comparing the width in  \cref{bigbox} shows that $\rho^P=(q)$ and  that  $\pi = (m^{c+1})$.  
  This implies that ${\rm max}_\succ(s_\nu\circ s_{(m)})$ is a hook partition whereas
   ${\rm max}_\succ   ( s_{\rho}	\circ s_{ (m^{c+1} )}	)$ has second row of width at least $ q(m-1)>1$, a contradiction.  
   \color{black}
 Therefore we can assume that  $\nu^M\neq (n)$.  
Then  \cref{bigbox} implies  that the first $m-1$   rows of $\pi^T$ are all equal to $n/q=c+1$ and  therefore  
 $\pi=((m-1)^{c+1}) + \pi'$ for some $\pi'\vdash c+1$.  
 In particular, $\pi_1-\pi_2 \leq c+1$.  
We now consider \cref{1}: the difference between the first and second rows of   
$ {\rm max}_{\succ}(s_\nu \circ s_\mu)$ is 
$$\left((m-1)n+\nu_1\right)-\nu_2$$
whereas the difference between the first and second rows of $ {\rm max}_{\succ}(s_\rho \circ s_\pi)$ is 
 less than or equal to $q\times (\pi_1-\pi_2  +1)	 = n+q$.     
  Therefore the necessary inequality 
$$ (m-1)n+\nu_1 -\nu_2 \leq n+q$$
implies that $m=2$ (since $q<n$).  For the remainder of the proof $\mu=(2)$ and $\pi= (1^{c+1}) + \pi'
\vdash 2(c+1)$ and therefore $\rho^P=\rho$ and $\nu^M=\nu$.     

We first consider the case $c>1$. Here we have that $\ell(\pi)=c+1>2$ and so 
the difference between the first and second rows of $ {\rm max}_{\succ}(s_\rho \circ s_\pi)$ is 
 $q\times (\pi_1-\pi_2)=q(\pi'_1-\pi'_2)\le q(1+c)=n $. On the other hand, for   $ {\rm max}_{\succ}(s_\nu \circ s_{(m)})=(n)+\nu$ the difference is at least~$n$. For equality, we require $\pi'=(c+1)$, that is $\pi=(c+1, 1^c)$. 
Then \cref{1} becomes $(n)+\nu=(q(c+1)+q, q^{c-1}) \sqcup \rho$ and we find $\nu=(q^c) \sqcup \rho$.
We now employ the dominance ordering to examine the case 
 $$ 
\pi =(c+2,1^c) \qquad \nu=(q^c) \sqcup \rho.
 $$

A necessary condition for   $  \PStd( (c+2,1^{c})^{\rho}, \alpha))\neq \emptyset $  
  is that 
$\alpha_1+\alpha_2\leq q(c+3)$.  To see this, simply note that if $\SSTS\in  \PStd( (c+2,1^{c})^{\rho}, \alpha))$, then 
$$\SSTS : [\rho] \to \SStd_\NN((c+2,1^{c}))$$
and the maximum number of entries equal to 1 or 2 in a semistandard Young tableaux of shape $(c+2,1^{c})$ is equal to 
$(c+2)+1=c+3$ (the sum of the lengths of the first and second rows of $(c+2,1^{c})$).  
Thus $p(\rho, (c+2,1^c),\alpha) =0$ for any $\alpha$ such that $\alpha_1+\alpha_2 >q(c+3)$ by \cref{PW}.  
We shall now construct  a plethystic tableau $\SSTS \in \PStd((2)^{(q^{c}) \sqcup \rho},\beta)$ with $\beta_1+\beta_2>q(c+3)$. This tableau will either be of maximal possible weight or there  exists another plethystic tableau of the same shape but of weight $\beta' \rhd \beta$; in either case,  for a partition  for $\gamma \in \{\beta,\beta'\}$, $0 \neq p((q^c) \sqcup \rho,(2),\gamma )$ whereas  $p(\rho,(c+2,1^c),\gamma ) =0$  (by \cref{PW}), 
 providing us with the necessary contradiction.
  Let $\SSTT\in \PStd((2)^{ (q^c) \sqcup \rho}, \beta)$ be the plethystic tableau such that 
 $$\SSTT(a,b)=
\begin{cases}
\gyoung(2;2) &\text{if $(a,b)$ is the least dominant (lowest) removable node of $(q^{c}) \sqcup \rho$} \\
\gyoung(1;a) &\text{otherwise}.
\end{cases}
$$ 
This tableau has weight $\beta$ with $\beta_1=q(c+2)-1$ and $\beta_2=q+2$ and so $\beta_1+\beta_2=   
q(c+3)+1$ as required.

Finally, we consider the case $c=1$. Here 
$\mu=(2)$ and $\pi\vdash 2(c+1)=4$ is either $(3,1)$ or $(2,2)$.
 In the $(2^2)$ case, comparing the widths of the partition on the left and right   of \cref{1} we see that $\nu_1=0$, a contradiction.  
 In the $(3,1)$ case, comparison of maximal terms again reveals that $\nu=(q)  \sqcup \rho$. Now
$$
s_{\rho}\circ s_{(3,1)}=
s_{\rho}\circ (s_{(1^2)}\circ s_{(2)})=
(s_{\rho}\circ  s_{(1^2)})\circ s_{(2)}.
$$ 
We observe that  ${\rm max}_\succ(s_\rho\circ s_{(1^2)})=(q) \sqcup \rho$, but $s_\rho\circ s_{(1^2)}$ is decomposable unless $\rho=(1^2)$ by \cref{homog}. For $\rho \neq (1^2)$, we deduce that 
$s_{(q) \sqcup \rho}\circ s_{(2)}$ is properly contained in $s_{\rho}\circ s_{(3,1)}$. Thus we have $q=2$, $\rho=(1^2)$ and $\nu=(2,1^2)$, as required. 
    \end{proof}

We may conjugate (applying \cref{conjugate}) to complete the  case where $\mu$ is linear.

\begin{cor} \label{one_col}
 Let $\mu,\nu,  \pi,\rho$ be  partitions of $m,n,p,q>1$ respectively. 
 We suppose that   $\mu=(1^m)$.  
   If 
 $$
 s_\nu\circ s_\mu = s_\rho \circ s_\pi
 $$
 then either $\nu=\rho$ and $\mu=\pi$ or we are in the exceptional case 
 $$s_{(2,1^2)}\circ s_{(1^2)}  = 
s_{(1^2)}\circ s_{(2,1^2)}.  $$     
 \end{cor}

Let $\mu,\nu,  \pi,\rho$ be arbitrary partitions of $m,n,p,q>1$ respectively.   
 We now consider what the condition
\begin{equation}\label{rowenaslabel}
 {\rm max}_{\succ}(s_\nu \circ s_\mu)
 =
  {\rm max}_{\succ}(s_\rho \circ s_\pi) 
\end{equation}
  tells us about this quadruple of partitions.  
  We first suppose that $\ell(\mu)=\ell(\pi)=k$.  This implies that $\ell(\nu)=\ell(\rho)=\ell$, say. Furthermore,
  \begin{align}\label{max}
\begin{split} &( n\mu_1,n\mu_2,\dots, n\mu_{k-1}, n\mu_k-n+\nu_1, \nu_2,  \dots, \nu_\ell)
\\=
 &( q\pi_1,q\pi_2,\dots, q\pi_{k-1}, q\pi_k-q+\rho_1, \rho_2,  \dots, \rho_\ell).
\end{split}\end{align}
We set $d=\gcd(n,q)$,  
$e=\gcd(m,p)$  and set $n=n'd$, $q=q'd$, $m=m'e$, $p=p'e$.   
Since 
$mn=pq$, we note that $m'n'ed=p'q'ed$ and so $m'n'=p'q'$.  
Since $m'$ and $p'$ are coprime, as are $n'$ and $q'$,  it follows that $m'=q'$ and $p'=n'$.  Thus 
$$
m=q'e \quad
n=n'd \quad 
q=q'd 
\quad 
p=n'e.
$$
From \cref{max}, we observe that $n\mu_i=\pi_i q$ implies $n'\mu_i = q' \pi_i$, and so we can set 
$\alpha_i:=\tfrac{\mu_i}{q'}
 =\tfrac{\pi_i} {n'}
 \in\NN
 $ for all $1\leq i \leq k-1$.  
Now, $\mu \vdash m=q'e$ and so the final row length satisfies
$$
\mu_k=q'e - \sum_{i=1}^{k-1} q' \alpha_i = q'
\underbrace{\left(e- \sum_{i=1}^{k-1}   \alpha_i \right) }_{\alpha_k}.
$$
We have a partition 
$(\alpha_1,\dots,\alpha_k)\vdash e$ with $q'\alpha=\mu$, and, in a similar fashion,  we deduce that  $n'\alpha=\pi$.   
Without loss of generality, we now assume that $n\geq q$.  We plug in our equalities 
$\pi=n'\alpha$ and $\mu=q'\alpha$ back into \cref{max} and to show  that
$$
\rho_i=\nu_i\textrm{ for $i \ge 2$ and }   \nu_1 = (n-q)+\rho_1.
$$
We immediately obtain the following corollary.  

\begin{cor}
 Let $\mu,\nu,  \pi,\rho$ be  partitions of $m,n,p,q>1$, respectively. 
 We suppose that   $\ell(\pi)=\ell(\mu)$.  If 
 $$
 s_\nu\circ s_\mu = s_\rho \circ s_\pi
 $$
 then  $\nu=\rho$ and $\mu=\pi$.  
\end{cor}
\begin{proof}
By the discussion above, we know that we are dealing with a quadruple
$$
 \mu=q'\alpha,
 \qquad
 \nu=\rho+(n-q),
 \qquad \pi=n'\alpha,	
 \qquad
 \rho	.	
$$
Comparing the width of the  partitions on the left  and right  of 
$$    {\rm max}_{\succ}(s_{\nu^M} \circ s_{\mu^T})=
{\rm max}_{\succ}(s_{\rho^P} \circ s_{\pi^T})$$
we deduce that  $\ell(\mu)n=\ell(\pi)q$.  Thus $n=q$, $\nu=\rho$, $q'=n'$ and thus 
$\mu=\pi$, as required.  
\end{proof}

We now consider the case 
 where the  lengths of the partitions $\mu$ and $\pi$ (and hence $\nu$ and $\rho$) differ.
 We suppose (without loss of generality) that $\ell(\mu)<\ell(\pi)$.  
 We set $\ell(\mu)=k$ and $\ell(\pi)=k+c$ for some $c\geq 1$.  Thus $\ell(\rho)+c=\ell(\nu)=\ell$, say.  
Observe that $
 {\rm max}_{\succ}(s_\nu \circ s_\mu)
 =
  {\rm max}_{\succ}(s_\rho \circ s_\pi) 
$ if and only if  the partitions 
$$\scalefont{0.9}\begin{array}{ccccccccccccccccccc}\label{max2}
\!\!\!\!\!(n\mu_1 &\dots& n\mu_{k-1}& n\mu_k-n+\nu_1& \nu_2&  \dots	&\nu_c			&\nu_{c+1}	&\nu_{c+2} &\dots & \nu_\ell)
\\
\!\!\!\!\!(q\pi_1&\dots& q\pi_{k-1}& q\pi_k 			&q\pi_{k+1}	&   \dots &q\pi_{k+c-1}	&q \pi_{k+c}-q+\rho_1	&\rho_2	&\dots  &		 \rho_{\ell-c}).
\end{array}
 $$
 coincide.  
We deduce that 
 \begin{equation}\label{use1}
\mu=q'(\alpha_1,\dots,\alpha_k),	    \qquad 
  \pi = n' (\alpha_1,\dots, \alpha_{k-1}) \sqcup  {(\pi_k,\dots, \pi_{k+c})} 
  \end{equation} 
for $  \alpha \vdash e  $,   $(\pi_k,\dots, \pi_{k+c})\vdash n'\alpha_k$  and 
\begin{equation}\label{use2} \nu=(q\pi_k - n(q'\alpha_{k}-1)) \sqcup q(\pi_{k+1},\dots, \pi_{k+c-1})\sqcup ( q(\pi_{k+c}-1)+\rho_1)\sqcup (\rho_2, \rho_3, \ldots \rho_{\ell-c}) 
  \end{equation}
and, in order for $\nu$ to be a partition, we need
    $$
q\pi_k - n(q'\alpha_{k}-1)  \geq q \pi_{k+1}
    $$
which, rearranging, gives 
   $$
q(\pi_k - \pi_{k+1})   \geq n(\mu_k-1).
    $$
We are now ready to complete our proof of Theorem A.

 \begin{thm} \label{Rowena}
 Let $\mu,\nu,  \pi,\rho$ be  partitions of $m,n,p,q>1$, respectively. 
 We suppose that both $\mu$ and $\pi$ are non-linear  and $\ell(\pi)>\ell(\mu)$.     If 
 $$
 s_\nu\circ s_\mu = s_\rho \circ s_\pi
 $$
 then   $\nu=\rho$ and $\mu=\pi$.
 \end{thm}

 \begin{proof}
We set $\ell(\mu)=k\geq 2$ and $\ell(\pi)=k+c$ for $c\geq 1$.  
We first see what can be deduced from $ {\rm max}_{\succ}(s_\nu \circ s_\mu)
 =
  {\rm max}_{\succ}(s_\rho \circ s_\pi) 
$.  
From \cref{use1} we have that
\begin{equation}\label{max1}\mu=q'(\alpha_1,\dots,\alpha_k)	    \qquad 
  \pi = n' (\alpha_1,\dots, \alpha_{k-1}) \sqcup  {(\pi_k,\dots, \pi_{k+c})} 
  \end{equation} 
for $  \alpha \vdash e  $ and   $(\pi_k,\dots, \pi_{k+c})\vdash n'\alpha_k$,
and, from \cref{use2},  we  
deduce that $|\rho|<|\nu|$ and so $q=q'd<n'd=n$ which implies $q'<n'$.  
From \cref{use1} this implies that $\mu_1=q'\alpha_1<n'\alpha_1=\pi_1$ in other words $\ell(\mu^T) < \ell(\pi^T)$.

 We now see what can be deduced from $ {\rm max}_{\succ}(s_{\nu^M} \circ s_{\mu^T})
 =
  {\rm max}_{\succ}(s_{\rho^P} \circ s_{\pi^T}) 
$.  
We have already concluded that $\ell(\mu^T)< \ell(\pi^T)$.  Therefore applying \cref{use1} (but with the partitions 
$\mu^T$, $\nu^M$, $\pi^T$ and $\rho^P$) we deduce that 
\begin{equation}\label{max2}\mu^T=q'(\beta_1,\dots, \beta_{\mu_1})
\qquad
\pi^T= n'	(\beta_1,\dots, \beta_{\mu_1-1}) \sqcup (\pi^T_{\mu_1},\dots, \pi^T_{\pi_1})  \end{equation}
for some $\beta\vdash e$ and 
$ (\pi^T_{\mu_1},\dots, \pi^T_{\pi_1})\vdash n'\beta_{\mu_1}$. 
 
From \cref{max1,max2} we deduce that $\mu$ can be built from boxes of size $q'\times q'$.  In other words,
$$
\mu=q'(\underbrace{\gamma_1,\gamma_1,\dots, \gamma_1}_{q'},\underbrace{ \gamma_2,\gamma_2,\dots, \gamma_2}_{q'}, \dots).
$$
for some $\gamma\vdash m /q'^{2}$.  Since $\gamma$ might have repeated parts, we write $\gamma$ in the form
$$\gamma=(a_1^{b_1},a_2^{b_2},\dots, a_x^{b_x})$$
where $a_1>a_2>\dots >a_x$, so
$$
\gamma^T=({(b_1+\dots+b_x )}^{a_x},{(b_1+\dots+b_{x-1})}^{a_{x-1}-a_x},\dots, {b_1}^{a_1-a_2 }).$$
Now,  \cref{max1} reveals  that
\begin{equation}\label{pi}
\pi = (\underbrace{n'a_1,\dots ,n'a_1}_{b_1q'},\underbrace{n'a_2,\dots ,n'a_2}_{b_2q'},\dots, 
\underbrace{n'a_x,\dots ,n'a_x}_{b_xq'-1},
\color{black}\pi_k,
\color{black} \dots , \pi_{k+c})
\end{equation} where $(\pi_k,
  \dots , \pi_{k+c})\vdash n'a_x$ and, from \cref{max2},
\begin{align}\label{pit}\begin{split}
\pi^T = &( ({n'{(b_1+\dots+b_x )}})^{ q' a_x},({n'{(b_1+\dots+b_{x-1} )}})^ { q'( a_x-a_{x-1})}, \dots 
 (n'b_1)^{q'(a_1-a_2)-1}) \sqcup 
 \\ &
( \pi^T_{\mu_1},  \dots ,  \pi^T_{\pi_1})
 \end{split}\end{align}
where 
$(\pi^T_{\mu_1}, \dots ,  \pi^T_{\pi_1})\vdash n'b_1$.  
 By looking at the first row of $\pi^T$ we deduce that provided $x\ne 1$ the last part of $\pi$ is $q'a_x$ and that it appears with multiplicity $n'b_x$.  This implies that 
$$
(\color{black}
\pi_k, \color{black}
\dots, \pi_{k+c})
=
(\dots, \underbrace{q'a_x,\dots q'a_x}_{n'b_x})\vdash n'a_x.
$$
But the sum over these final $n'b_x$ rows is 
$ 
q'a_a \times n' b_x 
$ 
which implies $q'=1$ and $b_x=1$ and that 
$$
(\color{black}
\pi_k, \color{black}
\dots, \pi_{k+c})
=
( \underbrace{ a_x, \dots , a_x}_{n' })\vdash n'a_x.  
$$
Now we input this into \cref{pi} to deduce that 
$$\ell(\pi)
=b_1+\dots +b_x -1+n'.$$ 
  On the other hand by \cref{pit} we know that
$$\ell(\pi)=\pi^T_1= 
 n' (b_1+\dots+b_x )
$$
Therefore 
$$
n'(b_1+\dots+b_x -1)
=
b_1+b_2+\dots +b_x-1  
$$
and thus $n'=1$ or $b_1+b_2+\dots +b_x =1$.  If $n'=1$ then $n=q$, contrary to our earlier observation that $q<n$.
If  $b_1+b_2+\dots +b_x =1$, then $\ell(\gamma)=\ell(\alpha)=\ell(\mu)=1$, contrary to our assumption that $\mu$ is non-linear.

Finally, it remains to consider the $x=1$ case.  This is the case in which $\gamma=(a^b)$ is a rectangle.  
Here we have that 
$ 
\mu=q'(a^{bq'}), \mu^T=q'(b^{aq'}) 
  $
and 
\begin{align}\label{1111}
 \pi = ((n'a)^{q'b-1})\sqcup( \pi_k,\dots ,\pi_{k+c}) \quad &\text{for $( \pi_k,\dots ,\pi_{k+c})\vdash n'a$}
\\ 		\label{2222}
\pi  =  ({(q'a-1)}^{n'b})+( \pi_{\mu_1}^T,\dots ,\pi_{\pi_1}^T)^T \quad
&\text{for $( \pi_{\mu_1}^T,\dots ,\pi_{\pi_1}^T)\vdash n'b$.  }
\end{align}
Now, recall that $q'<n'$; 
and so 
$$
q'b-1<q'b <n'b
$$
and so the  rectangle  in  \cref{1111} is at least 2 rows shorter than that in \cref{2222}.
This implies that $q'=1$ and $a$ or $b=1$ and so $\mu$ is linear, a contradiction.  
\end{proof}

We have now classified all possible equalities between  products $
 s_\nu\circ s_\mu = s_\rho \circ s_\pi
 $ where neither, one, or both of $\pi$ and $\mu$ are linear partitions.  
This completes the proof of Theorem A.  

\begin{ack}
We would like to thank Cedric Lecouvey for
 bringing to our attention the question of unique factorisability of products of Schur  functions and in particular for introducing us to  Rajan's  result for Littlewood--Richardson products.  
 
\end{ack}

\bibliographystyle{amsalpha}


 \end{document}